\documentclass[11pt,reqno]{amsart}
\usepackage{amsmath,amsthm,amssymb,mathrsfs}
\usepackage[abbrev,non-sorted-cites]{amsrefs}
\usepackage{hyperref}
\usepackage{color,graphicx}
\hypersetup{
  pdftitle={The Exact Comass Criterion and Calibrated Planes for the Tsai–Wang Form},
  pdfsubject={Exact comass, higher-rank sufficiency, and calibrated planes for Tsai-Wang Form},
  pdfkeywords={calibration, calibrated planes, comass, equality cases, minimal graph, product inequality, singular values}
}

\theoremstyle{plain}
\newtheorem{thm}{Theorem}[section]
\newtheorem{lem}[thm]{Lemma}
\newtheorem{prop}[thm]{Proposition}

\theoremstyle{definition}

\newtheorem{rmk}[thm]{Remark}
\newtheorem{ex}[thm]{Example}
\newtheorem*{AIuse}{AI Usage Statement}
\newtheorem*{ackn}{Acknowledgments}
\newtheorem*{assumption}{Assumption}
\numberwithin{equation}{section}

\newcommand\pl{\partial}

\newcommand\w{\wedge}
\newcommand\ld{\lambda}
\newcommand\Om{\mathcal{D}}
\newcommand\Ta{\Theta}
\newcommand\CS{\mathcal{S}}
\newcommand\CC{\mathcal C}
\newcommand\BR{\mathbb{R}}
\newcommand\dd{{\mathrm d}}
\newcommand\RGr{\mathrm{Gr}}
\newcommand\sdf{\mathrm{I}\!\mathrm{I}}
\DeclareMathOperator{\tr}{tr}
\DeclareMathOperator{\vol}{vol}
\DeclareMathOperator{\spn}{span}
\DeclareMathOperator{\rank}{rank}
\DeclareMathOperator{\graph}{graph}

\title{The Exact Comass Criterion\\for the Tsai--Wang Form}
\author{Wei-Chun Chen}
\address{Department of Mathematics, National Taiwan University, Taipei 10617, Taiwan}
\email{b14201028@ntu.edu.tw}

\author{Chun-Kai Lien}
\address{Department of Mathematics, National Taiwan University, Taipei 10617, Taiwan}
\email{d13221002@ntu.edu.tw}

\begin{document}

\begin{abstract}
For a map $F:\Om\subset\BR^n\to\BR^m$, let $\Ta(F)$ be the $n$-form on $\Om\times\BR^m$ introduced by Tsai and Wang \cite{TsaiWang}.  If $\ld_1,\ldots,\ld_{r}$ are the nonzero
singular values of $\dd F$ at any fixed point, we prove that
\[
 \Ta(F) \text{ has comass one if and only if }~
 \CS(\ld):=\sum_{i=1}^{r}\frac{\ld_i^2}{1+\ld_i^2}\leq1.
\]

Note that if \(\rank \dd F\le1\), \(\CS\) is always less than \(1\).  We also prove a dichotomy that if \(\CS\leq 1\), then either \(\CS < 1\) or \(\CS\equiv1\).
When \(\CS<1\), the graph tangent plane is the only calibrated plane, generalizing the corresponding result for hypersurfaces.
When \(\CS\equiv1\), we prove that \(F\) is affine or of rank \(2\).
\end{abstract}

\maketitle

\section{Introduction}\label{sec:introduction}

An $n$-form $\omega$ on a Euclidean space has comass
\[
\|\omega\|
=
\sup\{|\omega(\xi)|:\xi\text{ is a unit simple }n\text{-vector}\}.
\]
A closed form of comass at most one is a calibration in the sense of Harvey--Lawson \cite{HarveyLawson1982}, and an oriented submanifold is calibrated when the restriction of the form is its induced volume form. The key property is that calibrated submanifolds are minimizing.

 For minimal hypersurfaces in \(\mathbb{R}^{n+1}\) that are graphical over \(\mathbb{R}^n\times\{0\}\), it is well known that one can associate a calibration form. Thus, such hypersurfaces are automatically minimizing.  In higher codimensions, graphical minimal submanifolds need not be area-minimizing; see the seminal work of Lawson and Osserman \cite{LawsonOsserman1977}.

In \cite{TsaiWang}, Tsai and Wang associate to a graphical map
\[
 F:\Om\subset\BR^n\longrightarrow\BR^m
\]
an explicit $n$-form $\Ta(F)$ whose restriction to the graph is the induced volume form.  They also show that $\Ta(F)$ is closed precisely when the graph is minimal.  Let $\ld_1\geq\ldots\geq\ld_r>0$ be the nonzero\footnote{\(r\) is the rank of \(\dd F\).} singular values of $\dd F$ (at any fixed point).  They proved that if \(\ld_i\ld_j\leq\frac{\epsilon}{r-1}\) for any \(i\neq j\), then the comass of \(\Ta(F)\) is \(1\), where \(\epsilon\in(0,1)\) is an absolute constant.  They also pointed out that \(\ld_i\ld_j\leq\frac{1}{r-1}\) is a necessary\footnote{for conditions on \(2\)-dilations} condition for comass one.

It is therefore natural to ask for the \emph{sharp} condition
under which $\Ta(F)$ has comass one.
Our main result, Theorem~\ref{thm:main}, provides a complete answer.

Let $\Pi$ denote the orthogonal projection onto the vertical factor.
At each point of $\Gamma_F$, choose an orthonormal tangent basis
$e_1,\ldots,e_n$ and set
\begin{equation}\label{eq:S-geometric}
 \CS(\ld)
 :=\sum_{i=1}^n|\Pi e_i|^2
 =\sum_{i=1}^n\langle\Pi e_i,e_i\rangle
 =\sum_{i=1}^r\frac{\ld_i^2}{1+\ld_i^2}.
\end{equation}
This definition is independent of the choice of basis.
Using a smooth local orthonormal frame shows that $\CS$ is smooth,
even where the rank of $\dd F$ changes.

\begin{thm}\label{thm:main} 
Let \(F:\Om\subset\BR^n\longrightarrow\BR^m\) satisfy the minimal graph system.  At any $p=(x,y)\in\Om\times\BR^m$, \(\bigr\|\Ta(F)|_p\bigr\|=1\) if and only if \(\CS(\ld)\leq1\).
\end{thm}
In rank \(1\), \(\CS(\ld)<1\) automatically, reflecting the hypersurface case. In rank \(2\), the condition reduces to $\lambda_1\lambda_2\le1,$ which is precisely the area-non-increasing condition in \cite{TsaiWang}*{Theorem~1.4}.  

Section~\ref{sec:criterion} shows that, whenever $\CS(\ld)<1$,
the tangent plane of the graph is the unique calibrated plane.
When $r=2$ and $\lambda_1\lambda_2=1$, the calibrated planes form
a $\mathbb{CP}^1$-family.

The exact condition, \(\CS(\ld)\leq1\), sits naturally between familiar pairwise bounds.
For $r\geq2$,
\[
 \left\{ \max_{i<j}\ld_i\ld_j\leq\frac1{r-1} \right\}
 \quad\subset\quad
 \Bigl\{\CS(\ld)\leq1\Bigr\}
 \quad\subset\quad
 \left\{\max_{i<j}\ld_i\ld_j\leq1\right\}.
\]
The constant \(\frac{1}{r-1}\) is sharp for the first inclusion, as is seen
by taking \( \ld_1=\cdots=\ld_r=\frac1{\sqrt{r-1}} \). The last condition is precisely the area-non-increasing condition. Such conditions have played an
important role in the study of minimal graphs in higher codimension
\cite{Wang2003}, and subsequently in graphical mean curvature flow
and related rigidity questions
\cites{TsuiWang2004,SavasHalilajSmoczyk2014}. In rank at least three, the exact region is strictly smaller
than the area-non-increasing region.  We remark that the region where \(\CS(\ld)\leq 1\) is contained in the stability region identified by Lee and Tsui \cite{LeeTsui}*{Theorem A}.

Our next result investigate the condition \(\CS(\ld)\leq1\).
\begin{thm} \label{thm:intro2}
    Let \(F:\Om\subset\BR^n\longrightarrow\BR^m\) satisfy the minimal graph system, and \(\CS(\ld)\leq1\) on \(\Om\).  Then, either \(\CS(\ld)<1\) on \(\Om\) or \(\CS(\ld)\equiv 1\).
    Furthermore, if \(\CS(\ld)\equiv1\), there are two possibilities.
    \begin{itemize}
        \item \(\rank\dd F\geq 3\): \(F\) must be an affine function.
        \item \(\rank\dd F\equiv2\): the image of \(F\) is contained in a \(2\)-dimensional affine subspace of \(\BR^m\).
    \end{itemize}
\end{thm}

This paper is organized as follows.  Section~\ref{sec:reduced-form} derives the pointwise reduced form
and its graph-plane evaluation.    Section~\ref{sec:product} proves the sharp product inequality
and introduces a companion calibration form.
Section~\ref{sec:criterion} proves Theorem~\ref{thm:main}, and described the calibrated planes.
Finally, Section~\ref{sec:single-contact-rigidity} establishes Theorem~\ref{thm:intro2}. 

\begin{AIuse}
The authors formulated the problem of determining the exact comass-one condition and provided the relevant mathematical literature to ChatGPT 5.6 Sol. Through an extended, iterative interaction guided by the authors, ChatGPT 5.6 Sol identified the explicit sharp criterion and developed the main proof ideas of Theorem~\ref{thm:main}. The authors then checked, and substantially revised these arguments.
The authors take full responsibility for the correctness and content of the article.
\end{AIuse}

\begin{ackn}
    The authors thank Chung-Jun Tsai for helpful discussions and suggestions concerning this work, in particular for suggesting that we investigate the condition \(\CS\leq1\).
\end{ackn}

\section{The Differental Form Associated with a Map}\label{sec:reduced-form}

Let $F:\Om\subset\BR^n\to\BR^m$ be a smooth map,
and orient its graph by projection to $\Om$.  The $n$-form $\Ta(F)$
introduced in \cite{TsaiWang}*{Definition~2.4} restricts to the induced
volume form on the graph.  We use the coordinate expression from
\cite{TsaiWang}*{Proposition~3.1, equation~(3.2)}.

\begin{assumption}
The main purpose of Section~\ref{sec:reduced-form} through \ref{sec:criterion} (except in Section~\ref{sec:single-contact-rigidity}) is to evaluate the comass.  Since it is a pointwise calculation, we may assume $m=n$ and $\dd F$ has full rank at the point under consideration; see \cite{TsaiWang}*{Section~2.1}.  In particular, all singular values $\ld_1,\ldots,\ld_n$ are positive.

\end{assumption}

Fix $p=(x,y)$.  Choose oriented orthonormal coordinates
$x_1,\ldots,x_n$ on the domain and paired orthonormal coordinates
$y_1,\ldots,y_n$ on the target, without prescribing the orientation
of the target coordinates, so that
\begin{equation}\label{eqn:SVD-reduced}
 \dd F|_x\left(\frac{\pl}{\pl x_i}\right)
 =\ld_i\frac{\pl}{\pl y_i},
 \qquad i=1,\ldots,n,
\end{equation}
where $\ld_i>0$.  Set
\(\widehat{\dd x_i}
 =\dd x_1\w\cdots\w\dd x_{i-1}
  \w\dd x_{i+1}\w\cdots\w\dd x_n\),
and write
\begin{equation}\label{eq:coefficients}
 \CC(\ld)=\sqrt{\prod_{i=1}^{n}(1+\ld_i^2)},
 \qquad
 \CS(\ld)=\sum_{i=1}^{n}\frac{\ld_i^2}{1+\ld_i^2}.
\end{equation}
By \cite{TsaiWang}*{(3.2)}, in these coordinates, the \(n\)-form \(\Ta(F)\) takes the form:
\begin{equation}\label{eq:reduced-form}
 \Ta:=\Ta(F)|_p
 =\CC(\ld)\left(
    (1-\CS(\ld))\,\dd x_1\w\cdots\w\dd x_n
    +\sum_{i=1}^{n}(-1)^{i-1}
      \frac{\ld_i}{1+\ld_i^2}\,
      \dd y_i\w\widehat{\dd x_i}
  \right).
\end{equation}

For any linear transformation
$B:\BR^n\to\BR^n$, denote by $\Gamma_B$ its graph in $\BR^n\oplus\BR^n$,
oriented by projection to $\BR^n\oplus\{0\}$.
The $n$-plane $\Gamma_B$ has the ordered frame
\[
 \frac{\pl}{\pl x_j}
 +\sum_{i=1}^{n}B_{ij}\frac{\pl}{\pl y_i},
 \qquad j=1,\ldots,n.
\]

Since graphical (over $\BR^n\oplus\{0\}$) $n$-planes are dense in the unoriented Grassmannian $\RGr(n,\BR^n\oplus\BR^n)$, by continuity it suffices to estimate $|\Ta(\Gamma_B)|$ for any $n\times n$ matrix $B$.

A direct computation shows that
\begin{equation}\label{eq:graph-formula}
 \Ta (\Gamma_B)
 =\CC(\ld)\,
   \frac{1-\CS(\ld)+\sum_i\frac{\ld_i}{1+\ld_i^2} \cdot B_{ii}}
        {\sqrt{\det(I+B^TB)}}.
\end{equation}
Clearly, if $B_\ld=\operatorname{diag}(\ld_1,\ldots,\ld_n)$, we have $\Ta (\Gamma_{B_\ld})=1$ and hence $\|\Ta\|\geq1$ for every $\ld$.

\section{Inequalities in Matrix Algebra}\label{sec:product}

\subsection{A Sharp Product Inequality}

The comass estimate rests on the following inequality.

\begin{lem}\label{lem:product}
Let $n\geq0$, let $v_i\geq0$ satisfy $\sum_{i=0}^n v_i^2=1$,
and let $u_i\in\mathbb R$. Then
\begin{equation}\label{eq:product-ineq}
 \left(\sum_{i=0}^n v_i u_i\right)^2
 \leq \prod_{i=0}^n(1-v_i^2+u_i^2).
\end{equation}
Equality holds when
$(u_0,\ldots,u_n)=\pm(v_0,\ldots,v_n)$.
When $n\geq2$ and $v_i>0$ for every $i$,
these are the only equality cases.
\end{lem}

\begin{proof}
Consider the column vectors
\[
 v=(v_0,\ldots,v_n)^T,
 \qquad u=(u_0,\ldots,u_n)^T.
\]
By assumption, $v$ is a unit vector. Set
\[
 M=I+(u-v)v^T.
\]
By the matrix determinant lemma,
\begin{equation}\label{eq:rank-one-determinant}
 \det M=1+v^T(u-v)=v^Tu
 =\sum_{i=0}^n v_i u_i.
\end{equation}
Denote the $i$-th row of $M$ by $r_i$, and let $e_i$ be the
$i$-th standard basis column vector. Since
\[
 r_i=e_i^T+(u_i-v_i)v^T,
\]
we have
\[
 |r_i|^2
 =1+2v_i(u_i-v_i)+(u_i-v_i)^2
 =1-v_i^2+u_i^2.
\]
Hadamard's determinant inequality now yields
\[
 \left(\sum_{i=0}^n v_i u_i\right)^2
 =(\det M)^2
 \leq\prod_{i=0}^n |r_i|^2
 =\prod_{i=0}^n(1-v_i^2+u_i^2),
\]
which is \eqref{eq:product-ineq}. If $u=\pm v$, both sides
equal one, proving the equality assertion.

Suppose that $n\geq2$ and $v_i>0$ for every $i$.
Then $|r_i|^2>0$ for every $i$, so equality in Hadamard's
inequality holds exactly when the rows are pairwise orthogonal.
For $i\ne j$,
\begin{align*}
 0=\langle r_i,r_j\rangle
 &=v_j(u_i-v_i)+v_i(u_j-v_j)+(u_i-v_i)(u_j-v_j)\\
 &=u_i u_j-v_i v_j.
\end{align*}
Therefore $u_i u_j=v_i v_j$ for all $i\ne j$.
Since $n\geq2$ and all $v_i>0$, these relations force $u=\pm v$.
\end{proof}

\subsection{An Auxiliary Calibration}

We extend $\BR^n\oplus\BR^n$ by $\BR^1\oplus\BR^1$, denoting the two additional coordinate directions by $x_0$ and $y_0$. Later on, we will introduce an auxiliary $(n+1)$-form $\Phi$ on $\BR^{n+1}\oplus\BR^{n+1}$ associated with $\Ta$.  In this subsection, we work with a general expression of $\Phi$, and study its comass.

For $i=0,\ldots,n$, let $s_i\geq0$ and $c_i>0$ satisfy
$s_i^2+c_i^2=1$ and $\sum_{i=0}^n s_i^2=1$.
Define
\begin{equation}\label{eq:Phi-s}
 \Phi=
 \left(\prod_{j=0}^n c_j\right)^{-1} \left( s_i c_i\,\dd y_0\w\dd x_1\w \cdots \dd x_n  + 
 \sum_{i=1}^n(-1)^i s_i c_i\,\dd y_i\w\dd x_0\w\widehat{\dd x_i} \right).
\end{equation}

Recall von Neumann's trace inequality and its equality
characterization \cite{Carlsson2021}.
For $A,B\in\BR^{(n+1)\times(n+1)}$, write $\sigma_i(A)$ and $\sigma_i(B)$
for their singular values in \emph{nonincreasing order}.  Then
\begin{equation}\label{eq:von-neumann}
 \bigl|\tr(A^TB)\bigr|\leq\sum_{i=0}^n\sigma_i(A)\sigma_i(B).
\end{equation}
Equality holds if and only if, for some orthogonal matrices $U,V$,
\[
 A=U\operatorname{diag}(\sigma_i(A))V^T,
 \qquad
 B= \pm\,U\operatorname{diag}(\sigma_i(B))V^T.
\]

\begin{prop}\label{prop:Phi-comass}
The form $\Phi$ defined by \eqref{eq:Phi-s} has comass one.
\end{prop}

\begin{proof}
It suffices to prove that $|\Phi(\Gamma_B)|\leq1$ for every linear map $B:\BR^{n+1}\to\BR^{n+1}$.
Put $A=\operatorname{diag}(s_0c_0,\ldots,s_nc_n)$ and let
$b_0,\ldots,b_n$ be the singular values of $B$.
Relabel the pairs $(s_i,c_i)$ so that $s_ic_i$ is
nonincreasing, and list $b_i$ in nonincreasing order.
By \eqref{eq:graph-formula} and \eqref{eq:von-neumann},
\begin{align*}
 |\Phi(\Gamma_B)|
 &=\left(\prod_jc_j\right)^{-1}
   \frac{|\tr(AB)|}{\sqrt{\det(I+B^TB)}} \\
 &\leq\left(\prod_jc_j\right)^{-1}
   \frac{\sum_i s_ic_ib_i}{\prod_i\sqrt{1+b_i^2}} \\
 &\leq\left(\prod_jc_j\right)^{-1}
   \frac{\prod_i\sqrt{1-s_i^2+c_i^2b_i^2}}
        {\prod_i\sqrt{1+b_i^2}}
 =1,
\end{align*}
where the second inequality follows from Lemma~\ref{lem:product},
applied with $v_i=s_i$ and $u_i=c_ib_i$.

Note that $B=\operatorname{diag}(s_i/c_i)$ achieves the equality.
\end{proof}

\section{The Exact Comass Criterion}\label{sec:criterion}

This section is devoted to the proof of Theorem~\ref{thm:main}, using the notation of \eqref{eq:coefficients}--\eqref{eq:reduced-form}.

\subsection{Sufficiency}
Assume $\CS(\ld)\leq1$.  For $i=1,\ldots,n$, set
\begin{equation}\label{eq:sc-parameters}
 s_i=\frac{\ld_i}{\sqrt{1+\ld_i^2}},
 \qquad
 c_i=\frac1{\sqrt{1+\ld_i^2}}.
\end{equation}
Then $\sum_{i=1}^n s_i^2=\CS(\ld)$.
For the additional coordinate pair $x_0,y_0$, set
\begin{equation}
 s_0=\sqrt{1-\CS(\ld)},
 \qquad c_0=\sqrt{\CS(\ld)}.
\end{equation}
Thus
\begin{equation}\label{eq:unit-completion}
 \sum_{i=0}^{n} s_i^2=1,
 \qquad
 \prod_{i=0}^{n} c_i=c_0\CC(\ld)^{-1}.
\end{equation}
With these choices of $s_i$'s and  $c_i$'s, \eqref{eq:Phi-s} becomes
\begin{equation}\label{eq:Phi-tilde}
 \Phi
 = \CC(\ld)s_0\,\dd y_0\w
   \dd x_1\w\cdots\w\dd x_n
 + \frac{\CC(\ld)}{c_0}
   \sum_{i=1}^{n}(-1)^i s_i c_i\,
   \dd y_i\w\dd x_0\w\widehat{\dd x_i},
\end{equation}
and Proposition~\ref{prop:Phi-comass} asserts that $\Phi$ has comass one.

Let
\[
 W_0=c_0\frac{\pl}{\pl x_0}+s_0\frac{\pl}{\pl y_0}.
\]
The key observation is the contraction identity
\begin{equation}\label{eq:contraction-identity}
 \left.(\iota_{W_0}\Phi)\right|_{\BR^n\oplus\BR^n}=\Ta.
\end{equation}
Since $W_0$ is unit and orthogonal to $\BR^n\oplus\BR^n$,
$\|\Ta\|\leq\|\Phi\|=1$.

\subsection{Necessity}

By \eqref{eq:graph-formula}, for $B=\operatorname{diag}(-\ld_1,\ldots,-\ld_n)$, we have
\[
 \Ta(\Gamma_{B})=1-2\CS(\ld).
\]
If $\CS(\ld)>1$, $\Ta(\Gamma_{B}) < -1$.
This proves necessity and completes the proof of Theorem~\ref{thm:main}.

\subsection{Calibrated Planes} \label{sec:calibrated-planes}

For an $n$-form $\omega$ of comass one on $\BR^{2n}$, write
\[
 \mathcal G^+(\omega)
 :=\{P\in\RGr^+(n,\BR^{2n}):\omega(P)=1\}
\]
for its set of calibrated oriented planes. We first determine equality
for the corresponding $(n+1)$-form $\Phi$, then use the contraction
identity to prove uniqueness for $\Ta$ in the strict region.

\subsection{Calibrated Planes of \texorpdfstring{$\Phi$}{Phi}}\label{sec:cali-comp-form}

Throughout this subsection, assume that $n\geq2$, $0<s_i,c_i<1$ satisfy
\[
 s_i^2+c_i^2=1,
 \qquad
 \sum_{i=0}^n s_i^2=1,
\]
and let $\Phi$ be the $(n+1)$-form defined in \eqref{eq:Phi-s}.

\begin{lem}
\label{lem:von-neumann-equality}
Let $D=\operatorname{diag}(s_0c_0,\ldots,s_nc_n)$.
Suppose that, after the common relabeling used in
\eqref{eq:von-neumann}, a matrix $B$ has singular values $b_i=s_i/c_i$
and satisfies
\begin{equation}
 |\tr(DB)|=\sum_{i=0}^n s_ic_i b_i.
\end{equation}
Then, $B$ must be $\operatorname{diag}(s_0/c_0,\ldots,s_n/c_n)$ or $-\operatorname{diag}(s_0/c_0,\ldots,s_n/c_n)$.
\end{lem}

\begin{proof}
By the equality case of \eqref{eq:von-neumann}, $D$ and $\pm B$ admit a common ordered singular value decomposition.  For simplicity, assume it is $B$.
If $s_ic_i=s_jc_j$ for some $i\neq j$, then
\[
 0=s_i^2c_i^2-s_j^2c_j^2
  =(s_i^2-s_j^2)(1-s_i^2-s_j^2).
\]
Since $n\geq2$, all $s_k>0$, and $\sum_{k=0}^n s_k^2=1$, the second factor, $1-s_i^2-s_j^2$, is positive.  Thus $s_i=s_j$ and $c_i=c_j$, so the singular values
$b_i=s_i/c_i$ are constant on every repeated singular-value block of $D$.
The choice of common singular frames within these blocks therefore
does not affect $B$, and hence $B = \operatorname{diag}(s_0/c_0,\ldots,s_n/c_n)$.
\end{proof}

\begin{lem}
\label{lem:nontransverse-calibrated}
If an oriented $(n+1)$-plane $P$ is not transverse to
$\{0\}\oplus\BR^{n+1}$, then
$P\notin\mathcal G^+(\Phi)$.
\end{lem}

\begin{proof}
For such a $P$, choose a unit vector
\[
 \nu = \sum_{i=0}^n a_i\frac{\pl}{\pl y_i}
 \in P\cap(\{0\}\oplus\BR^{n+1}),
\]
and set
\[
 \vol_x
 =\dd x_0\w\cdots\w\dd x_n,
 \qquad
 w=\sum_{i=0}^n
 s_ic_ia_i\frac{\pl}{\pl x_i}.
\]
Contracting \eqref{eq:Phi-s} with $\nu$ gives
\[
 \iota_\nu\Phi=\left(\prod_jc_j\right)^{-1}\iota_w\vol_x.
\]
Since $\|\iota_w\vol_x\|=|w|$, we obtain
\begin{align*}
    \|\iota_\nu\Phi\|
    &=\left(\prod_jc_j\right)^{-1} \left(\sum_{i=0}^n s_i^2c_i^2a_i^2\right)^{1/2}
    = \left( \sum_{i=0}^n \left(\frac{s_i}{\prod_{j\ne i}c_j}\right)^2a_i^2 \right)^{1/2} \leq\max_i\frac{s_i}{\prod_{j\ne i}c_j}.
\end{align*}
Since $n\geq2$ and $s_j>0$, for every $i$,
\[
 \prod_{j\ne i}c_j^2
 =\prod_{j\ne i}(1-s_j^2)
 >1-\sum_{j\ne i}s_j^2=s_i^2.
\]
It follows that $\|\iota_\nu\Phi\|<1$ and thus $P\notin\mathcal G^+(\Phi)$.
\end{proof}

\begin{lem}
\label{lem:one-normal-calibrated}
 Let $\Phi$ be the form in
\eqref{eq:Phi-s}.  Then
\[
 \mathcal G^+(\Phi)
 =\{Q_+,Q_-\},
\]
where
\begin{align*}
 Q_+&=\spn^+
   \left\{c_i\frac{\pl}{\pl x_i}+s_i\frac{\pl}{\pl y_i}:
   0\leq i\leq n\right\},\\
 Q_-&=-\spn^+
   \left\{c_i\frac{\pl}{\pl x_i}-s_i\frac{\pl}{\pl y_i}:
   0\leq i\leq n\right\}.
\end{align*}
Here the minus sign in front of the second oriented span reverses its
orientation.
\end{lem}

\begin{proof}
By Lemma~\ref{lem:nontransverse-calibrated}, any \(P\in\mathcal G^+(\Phi)\) must be \(\pm\Gamma_B\)
for some linear map $B:\BR^{n+1}\to\BR^{n+1}$. Since \(|\Phi(\Gamma_B)|=1\), every inequality in the proof of
Proposition~\ref{prop:Phi-comass} is an equality.

Writing \(b_i\) for the singular values of \(B\), equality in
Lemma~\ref{lem:product} gives
\[
b_i=\frac{s_i}{c_i}.
\]
Lemma~\ref{lem:von-neumann-equality} therefore yields
\[
B=\pm\operatorname{diag}\!\left(\frac{s_i}{c_i}\right).
\]
For the positive sign, \(\Gamma_B=Q_+\) and \(\Phi(\Gamma_B)=1\), so \(P=Q_+\). For the negative sign,
\(\Phi(\Gamma_B)=-1\), so \(P=Q_-\).
\end{proof}

\subsection{Calibrated Planes of \texorpdfstring{$\Ta$}{Theta}}\label{sec:characterization}

Write $A=\dd F|_x$ and $\Gamma_A=\graph A$, oriented by projection to the domain.

\begin{prop}\label{prop:strict-calibrated-plane}
If $\CS(\ld)<1$, then $\Gamma_A$ is the unique calibrated
oriented $n$-plane of $\Ta$.
\end{prop}

\begin{proof}
The assertion is easy when $n=1$. Assume that $n\geq2$, and consider
the $(n+1)$-form $\Phi$ from Section~\ref{sec:criterion}. If $P$ is a
calibrated $n$-plane, then \eqref{eq:contraction-identity} gives
\[
 \Phi(W_0\w P)=\Ta(P)=1.
\]
Here $\CS(\ld)<1$ implies $s_0>0$, as required in
Lemma~\ref{lem:one-normal-calibrated}.

By Lemma~\ref{lem:one-normal-calibrated}, $W_0\w P$ is either $Q_+$
or $Q_-$, where $Q_+=W_0\w \Gamma_A$.
The intersection of $Q_-$ with the $(x_0,y_0)$-plane is spanned by
\[
 c_0\frac{\pl}{\pl x_0}-s_0\frac{\pl}{\pl y_0}.
\]
Since $s_0,c_0>0$, this line does not contain $W_0$.
Thus $W_0\w P$ cannot be $Q_-$, and consequently
$W_0\w P=Q_+=W_0\w\Gamma_A$. Hence $P=\Gamma_A$.
\end{proof}

\begin{rmk}\label{rem:high-rank-boundary}
A careful examination of the equality cases in the preceding estimates
shows that, when $n\geq3$ and $\CS(\ld)=1$, the only calibrated
planes of $\Ta$ are $\Gamma_A$ and $-\Gamma_{-A}$.
\end{rmk}

\begin{rmk}\label{rem:rank-two-boundary}
If $n=2$ and $\CS(\ld)=1$, then $\ld_1\ld_2=1$.
Thus $c_2=s_1$ and $c_1=s_2$, and the form becomes
\[
 \Ta=\dd y_1\w\dd x_2-\dd y_2\w\dd x_1
 =\dd x_1\w\dd y_2-\dd x_2\w\dd y_1;
\]
see also \cite{TsaiWang}*{(2.12)}.
After an orthogonal change of coordinates, this is the standard K\"ahler calibration on $\mathbb{R}^4$, and its calibrated planes are precisely the complex lines.
\end{rmk}

\section{A Rigidity Phenomenon} \label{sec:single-contact-rigidity}

Let $\Om\subset\BR^n$ be an open and connected domain, and let
$F:\Om\to\BR^m$ be a smooth map whose
graph $\Gamma_F$ is minimal. We use $\CS$ and $\Pi$ from \eqref{eq:S-geometric}.
All derivatives below use the induced metric on $\Gamma_F$, with
$\Delta=\operatorname{div}\nabla$.
For orthonormal frames of the tanget and normal
$e_i,e_\alpha$, write $\sdf=\{h_{\alpha ij}\}_{n+1\leq\alpha\leq n+m,1\leq i,j\leq n}$ for the
second fundamental form.

The following identities follow from the parallel two-tensor formulas of
Tsui and Wang \cite{TsuiWang2004}*{Proposition~3.2 and p.~1121},
applied to $\langle\Pi\cdot,\cdot\rangle$ on the minimal
graph and traced over the tangent space.

\begin{lem}\label{lem:rigidity-differential-identities}
Suppose that $\CS\leq1$ on $\Om$.
If $\CS(x_0)=1$ at some $x_0\in\Om$, then
$\CS\equiv1$.
\end{lem}
\begin{proof}
The parallel two-tensor formulas of Tsui and Wang
\cite{TsuiWang2004}*{Proposition~3.2 and p.~1121},
applied to $\langle\Pi\cdot,\cdot\rangle$ on the minimal
graph and traced over the tangent space, give
\begin{equation}\label{eq:rigidity-projection}
\begin{aligned}
 \nabla_k\CS
 &=2\sum_{\alpha,i}
   \langle\Pi e_\alpha,e_i\rangle h_{\alpha ik},\\
 \Delta\CS
 &=2\sum_{\alpha,\beta,i,k}
   \langle\Pi e_\alpha, e_\beta\rangle
   h_{\alpha ik}h_{\beta ik}
   -2\sum_{\alpha,i,j,k}
   \langle\Pi e_i,e_j\rangle
   h_{\alpha ik}h_{\alpha jk}.
\end{aligned}
\end{equation}

At any fixed point, let $r=\rank \dd F$ and choose an SVD frame, with nonzero singular values $\lambda_1,\ldots,\lambda_r$.  For \(i > r\), set $\lambda_i=0$.
In this frame, \eqref{eq:rigidity-projection} becomes
\begin{equation}\label{eq:rigidity-svd}
 \nabla_k\CS
 =2\sum_{i=1}^r
   \frac{\lambda_i}{1+\lambda_i^2}\,h_{(n+i)ik},
 \qquad
 \Delta\CS
 =2\sum_{\alpha,i,k}
   \left(
     \frac{1}{1+\lambda_{\alpha - n}^2}
     -\frac{\lambda_i^2}{1+\lambda_i^2}
   \right)h_{\alpha ik}^2.
\end{equation}
We denote by $\ast\Omega$ the Jacobian of the projection
from the graph $\Gamma_F$ onto $\Om$. Then
\[
 \ast\Omega
 =\bigl(\det(I+(\dd F)^T\dd F)\bigr)^{-1/2}
 =\prod_{i=1}^r(1+\lambda_i^2)^{-1/2}>0.
\]
By the computation in \cite{wang2002long}*{Section~3},
\begin{equation}\label{eq:graph-jacobian-gradient}
 \nabla_k\log(\ast\Omega)
 =-\sum_{i=1}^r\lambda_i h_{(n+i)ik}.
\end{equation}
Together with \eqref{eq:rigidity-svd}, this gives
\begin{align} \label{eqn:gradient-terms}
    \sum_{i=1}^r\frac{\lambda_i^3}{1+\lambda_i^2}h_{(n+i)ik} = \sum_{i=1}^r\lambda_i h_{(n+i)ik} -\sum_{i=1}^r\frac{\lambda_i}{1+\lambda_i^2}h_{(n+i)ik} =-\nabla_k\log(\ast\Omega)-\frac12\nabla_k\CS.
\end{align}

For \eqref{eq:rigidity-svd} of \(\Delta\CS\), we split the summation into two parts: \(\alpha = n+i\) (for \(i=1,\ldots,r\)), and \(\alpha\neq n+i\).  If \(\alpha\neq n+i\), note that
\begin{align} \label{eqn:unpaired-coefficient}
    \frac{1}{1+\lambda_{\alpha-n}^2} - \frac{\lambda_i^2}{1+\lambda_i^2}
    = 1-\frac{\lambda_{\alpha-n}^2}{1+\lambda_{\alpha-n}^2} - \frac{\lambda_i^2}{1+\lambda_i^2} .
\end{align}

For \(\alpha = n+i\),
\begin{align}
    &\quad \sum_{k=1}^n\biggl( \sum_{i=1}^r\frac{h_{(n+i)ik}^2}{1+\lambda_i^2} - \sum_{i=1}^r \frac{\lambda_i^2h_{(n+i)ik}^2}{1+\lambda_i^2} \biggr) \notag \\
    &= (1-\CS)\sum_{k=1}^n\sum_{i=1}^r \frac{h_{(n+i)ik}^2}{1+\lambda_i^2} + \sum_{k=1}^n\sum_{i=1}^r \frac{\lambda_i^2}{1+\lambda_i^2} \biggl( \sum_{j=1}^r\frac{h_{(n+j)jk}^2}{1+\lambda_j^2} - h_{(n+i)ik}^2 \biggr) \notag \\
    \begin{split} \label{eqn:paired-terms}
    &= (1-\CS) \sum_{k=1}^n\sum_{i=1}^r \frac{h_{(n+i)ik}^2}{1+\lambda_i^2} - \sum_{k=1}^n\biggl( \sum_{i=1}^r \frac{\lambda_i^3}{1+\lambda_i^2}h_{(n+i)ik} \biggr)\nabla_k\CS\\
    &\quad + \sum_{k=1}^n\sum_{i=1}^r \frac{\lambda_i^2}{1+\lambda_i^2} \biggl( \sum_{j=1}^r\frac{h_{(n+j)jk}^2}{1+\lambda_j^2} - h_{(n+i)ik}^2 + \lambda_i h_{(n+i)ik}\nabla_k\CS \biggr) .
    \end{split}
\end{align}
By \eqref{eqn:gradient-terms}, the second term in \eqref{eqn:paired-terms} is
\begin{align} \label{eqn:paired-part2}
    - \sum_{k=1}^n\biggl( \sum_{i=1}^r \frac{\lambda_i^3}{1+\lambda_i^2}h_{(n+i)ik} \biggr)\nabla_k\CS
    &= \langle\nabla\log(\ast\Omega),\nabla\CS\rangle + \frac12 |\nabla\CS|^2 .
\end{align}
For the last term in \eqref{eqn:paired-terms}, for any fixed \(i\in\{1,\ldots,r\}\) and \(k\in\{1,\ldots,n\}\),
\begin{align*}
    0 &\leq \sum_{\substack{j=1\\j\neq i}}^r \frac{1}{1+\lambda_j^2} \bigl( h_{(n+j)jk}^2+2\lambda_i\lambda_j h_{(n+i)ik}h_{(n+j)jk} + \lambda_i^2\lambda_j^2h_{(n+i)ik}^2 \bigr)\\
    &= \sum_{j=1}^r\frac{h_{(n+j)jk}^2}{1+\lambda_j^2} - \frac{h_{(n+i)ik}^2}{1+\lambda_i^2} + 2\lambda_i h_{(n+i)ik} \sum_{j=1}^r\frac{\lambda_j}{1+\lambda_j^2}h_{(n+j)jk} - \frac{2\lambda_i^2}{1+\lambda_i^2}h_{(n+i)ik}^2\\
    &\quad + \lambda_i^2(\CS-1)h_{(n+i)ik}^2 + \lambda_i^2 \left(1-\frac{\lambda_i^2}{1+\lambda_i^2}\right) h_{(n+i)ik}^2.
\end{align*}
It follows that
\begin{align} \label{eqn:paired-part3}
    \sum_{j=1}^r\frac{h_{(n+j)jk}^2}{1+\lambda_j^2} - h_{(n+i)ik}^2 + \lambda_i h_{(n+i)ik}\nabla_k\CS
    &\geq (1-\CS)\lambda_i^2h_{(n+i)ik}^2 .
\end{align}

By using \eqref{eqn:unpaired-coefficient}, \eqref{eqn:paired-terms}, \eqref{eqn:paired-part2} and \eqref{eqn:paired-part3}, \eqref{eq:rigidity-svd} leads to
\begin{align} \label{eqn:S-general} \begin{split}
    \frac12\Delta\CS &\geq (1-\CS) \sum_{k=1}^n\sum_{i=1}^r \frac{1+\lambda_i^4}{1+\lambda_i^2}h_{(n+i)ik}^2 + \sum_{\substack{\alpha,i,k\\ \alpha\neq n+i}} \Bigl( 1-\frac{\lambda_{\alpha-n}^2}{1+\lambda_{\alpha-n}^2} - \frac{\lambda_i^2}{1+\lambda_i^2} \Bigr) h_{\alpha ik}^2 \\
    &\quad + \langle\nabla\log(\ast\Omega),\nabla\CS\rangle + \frac12 |\nabla\CS|^2.
\end{split} \end{align}
If \(\CS\leq 1\), for \(\alpha\neq n+i\),
\begin{equation} \label{eqn:CoefficientGreaterThanZero}
    1-\frac{\lambda_{\alpha-n}^2}{1+\lambda_{\alpha-n}^2} - \frac{\lambda_i^2}{1+\lambda_i^2}
    \geq 1-\CS \geq 0.
\end{equation}
It follows that \(\frac12\Delta\CS \geq \langle\nabla\log(\ast\Omega),\nabla\CS\rangle + \frac12 |\nabla\CS|^2\).  If \(\CS\) achieves \(1\) in some interior point, the strong maximum principle \cite{GilbargTrudinger2001}*{Chapter~3}, applied to \eqref{eqn:S-general}, gives $\CS\equiv1$.
\end{proof}

With on Lemma~\ref{lem:rigidity-differential-identities}, we study what happens if \(\CS\equiv1\) on \(\Om\).  At where \(\rank \dd F \leq 1\), \(\CS < 1\).  Hence, there are two possibilities for \(\CS \equiv 1\): either \(\rank \dd F \geq 3\) somewhere, or \(\rank \dd F \equiv 2\).

\begin{prop} \label{thm:single-contact-rigidity} 
Suppose $\CS\equiv1$ on $\Om$.
If $\rank \dd F(x_0)\geq3$ at some $x_0\in\Om$, then $F(x)$ must be an affine function.
\end{prop}

\begin{proof}
By Lemma~\ref{lem:rigidity-differential-identities}, $\CS \equiv 1$ on $\mathcal D$. 
Since $\rank \dd F(x_0)\geq3$, $\rank \dd F(x) > 2$ on an open neighborhood of $x_0$.  By the unique continuation principle of Aronszajn
\cite{Aronszajn1957}, it suffices to show that $F$ is affine
on this neighborhood.  With this understood, we may assume $\rank \dd F(x)\geq3$ for all \(x\in\Om\).

At each point, choose an SVD frame as in
\eqref{eq:rigidity-svd}. Since $\nabla \CS = \Delta \CS = 0$, \eqref{eqn:S-general} and \eqref{eqn:CoefficientGreaterThanZero} give 
\[
0 \ge \sum_{\substack{\alpha,i,k\\ \alpha\neq n+i}} \Bigl( 1-\frac{\lambda_{\alpha-n}^2}{1+\lambda_{\alpha-n}^2} - \frac{\lambda_i^2}{1+\lambda_i^2} \Bigr) h_{\alpha ik}^2 \ge 0.
\]
Since $\rank \dd F \ge 3$, every coefficient in this sum is strictly positive. Consequently,
\[
    h_{\alpha i k} = 0
 \qquad\text{whenever }\alpha\ne n+i.
\]
By the symmetric of the second fundamental form, all the components vanish, except \(h_{(n+i)ii}\)'s.
By minimality,
\[
    h_{(n+i)ii} = -\sum_{\substack{1\leq j\leq n\\ j\neq i}} h_{(n+i)jj} = 0.
\]
Thus, \(\sdf\equiv0\), and \(\Gamma_F\) must be flat.
\end{proof}

\begin{prop}
\label{cor:rank-two-single-contact}
Suppose $\CS\equiv1$ and \(\rank \dd F \equiv 2\) on $\Om$.  Then, $F$ is a submersion onto an affine two-plane \(P\) in $\BR^m$; in other words, \(\Gamma_F\subset\BR^n\times P\cong\BR^{n+2}\).
\end{prop}

\begin{proof}
When $\rank \dd F=2$,
\[
    1 - \CS = \frac{1-\lambda_1^2\lambda_2^2}{(1+\lambda_1^2)(1+\lambda_2^2)}.
\]
Thus, $\lambda_1\lambda_2\equiv1$.

At each point, choose an SVD frame.
Since $\nabla\CS=\Delta\CS=0$, \eqref{eqn:S-general}  and \eqref{eqn:CoefficientGreaterThanZero} give 
\[
0 = \sum_{\substack{\alpha,i,k\\ \alpha\neq n+i}} \Bigl( 1-\frac{\lambda_{\alpha-n}^2}{1+\lambda_{\alpha-n}^2} - \frac{\lambda_i^2}{1+\lambda_i^2} \Bigr) h_{\alpha ik}^2 .
\]
By \eqref{eqn:CoefficientGreaterThanZero}, all terms are nonnegative, and the coefficients
of $h_{\alpha ik}^2$ with $\alpha> n+2$ are strictly positive.
Therefore $h_{\alpha ik}=0$ for every $\alpha>n+2$.

Geometrically, the vectors in \(\BR^m\) that are orthogonal to the image of \(\dd F\) constitutes an \(\BR^{m-2}\) vector bundle over \(\Gamma_F\).  Choose a local orthonormal frame \(e_{n+3},\ldots,e_{n+m}\) of it.  The condition $h_{\alpha ik}=0$ for every $\alpha>n+2$ implies that \(e_{n+3},\ldots,e_{n+m}\) are constant vectors, and this corollary follows.
\end{proof}

The map in Corollary~\ref{cor:rank-two-single-contact}
need not be affine, even though its image lies in an
affine two-plane.
The following example is a local graphical realization
of the complex-parabola model discussed by Joyce
\cite{Joyce2001}*{Section~10, discussion following Theorem~10.3}.
\begin{ex}
\label{ex:rank-two-nonaffine}

On a neighborhood of the origin in $\BR^2$, consider
\[
 F(x,y)=\left(1-\sqrt{1+x^2-2y},\,-x\sqrt{1+x^2-2y}\right).
\]
Direct differentiation gives $\det \dd F=1$ and
$\pl_x^2F(0,0)=(-1,0)$.
Thus $\rank \dd F=2$, $\lambda_1\lambda_2=\CS=1$,
and $F$ is not affine.

With $x=s$ and $y=t+\tfrac12(s^2-t^2)$,
its graph is parametrized by
\[
 \Psi(s,t)
 =\left(s,\,t+\tfrac12(s^2-t^2),\,t,\,s(t-1)\right).
\]
Writing the ambient coordinates as $(x,y,u,v)$ and setting
$z_1=x+iu$, $z_2=y+iv$, we obtain
\[
 z_1=s+it,\qquad
 z_2-\frac12=\frac12(z_1-i)^2.
\]
Thus the graph is a holomorphic curve, specifically
a translated complex parabola, and hence is minimal.
\end{ex}

\end{document}